\documentclass[11pt]{amsart}   	

\usepackage[margin=1in]{geometry}

\usepackage{amsmath,amssymb, amsfonts,amsthm,stackrel}
\usepackage[hide links]{hyperref}
\usepackage{mathtools}
\usepackage{tikz-cd}
\usetikzlibrary{decorations.markings}
\usepackage{makecell}
\setcellgapes{4pt}
\usepackage{multicol}

\usepackage[shortlabels]{enumitem}

\newtheorem{theorem}{Theorem}[section]
\newtheorem{lemma}[theorem]{Lemma}
\newtheorem{corollary}[theorem]{Corollary}
\newtheorem{proposition}[theorem]{Proposition}
\newtheorem{definition-proposition}[theorem]{Definition-Proposition}

\theoremstyle{definition}
\newtheorem{definition}[theorem]{Definition}

\newtheorem{notation}[theorem]{Notation}

\newtheorem{remark}[theorem]{Remark}

\DeclareMathOperator{\add}{\mathsf{add}}

\DeclareMathOperator{\Gen}{\mathsf{Gen}}

\DeclareMathOperator{\Cogen}{\mathsf{Cogen}}

\DeclareMathOperator{\wide}{\mathsf{wide}}

\DeclareMathOperator{\WL}{\mathsf{W_L}}

\newcommand{\J}{\mathsf J}
\newcommand{\Jinv}{\J^\mathsf{d}}
\newcommand{\T}{\mathcal T}

\newcommand{\F}{\mathcal F}
\newcommand{\G}{\mathcal G}
\newcommand{\W}{\mathcal W}
\newcommand{\C}{\mathcal C}
\newcommand{\B}{\mathcal B}
\newcommand{\A}{\mathcal A}
\newcommand{\M}{\mathcal M}
\renewcommand{\H}{\mathcal H}

\DeclareMathOperator{\tors}{\mathsf{tors}}
\DeclareMathOperator{\ftors}{\mathsf{f}\text{-}\mathsf{tors}}
\DeclareMathOperator{\torf}{\mathsf{torf}}
\DeclareMathOperator{\ftorf}{\mathsf{f}\text{-}\mathsf{torf}}
\DeclareMathOperator{\ice}{\mathsf{ice}}

\DeclareMathOperator{\fwide}{\mathsf{f}\text{-}\mathsf{wide}}

\DeclareMathOperator{\mods}{\mathsf{mod}}

\DeclareMathOperator{\Hom}{\mathrm{Hom}}
\DeclareMathOperator{\Ext}{\mathrm{Ext}}
\DeclareMathOperator{\im}{\mathrm{image}}
\DeclareMathOperator{\coker}{\mathrm{coker}}

\DeclareMathOperator{\I}{\mathcal{I}}
\DeclareMathOperator{\Is}{\mathcal{I}_\mathrm{s}}

\newcommand{\Imod}{I}
\newcommand{\Imods}{I_\mathrm{s}}

\newcommand{\cov}{\mathrm{cov}}

\newcommand{\sI}{\mathsf{I}}
\newcommand{\sL}{\mathsf{L}}

\newcommand{\derD}{\mathcal{D}^b}

\usepackage{soul}

\author{Eric J. Hanson}
\thanks{The author is supported by an AMS-Simons travel grant.}

\address{Department of Mathematics, North Carolina State University, Raleigh NC 27695, USA}
\email{ejhanso3@ncsu.edu}



\title[ICE-sequences and $\tau^{-1}$-rigid modules]{Sequences of ICE-closed subcategories via preordered $\tau^{-1}$-rigid modules}

\date{\today}

\begin{document}

\begin{abstract}
   Let $\Lambda$ be a finite-dimensional basic algebra. Sakai recently used certain sequences of image-cokernel-extension-closed (ICE-closed) subcategories of finitely generated $\Lambda$-modules to classify certain (generalized) intermediate $t$-structures in the bounded derived category. We classifying these ``contravariantly finite ICE-sequences'' using concepts from $\tau$-tilting theory. More precisely, we introduce ``cogen-preordered $\tau^{-1}$-rigid modules'' as a generalization of (the dual of) the ``TF-ordered $\tau$-rigid modules'' of Mendoza and Treffinger. We then establish a bijection between the set of cogen-preordered $\tau^{-1}$-rigid modules and certain sequences of intervals of torsion-free classes. Combined with the results of Sakai, this yields a bijection with the set of contravariantly finite ICE-sequences (of finite length), and thus also with the set of $(m+1)$-intermediate $t$-structures whose aisles are homology-determined. 
\end{abstract}

\maketitle

\tableofcontents


\section{Introduction}\label{sec:intro}

Let $\A$ be an abelian category. In \cite[Section~4]{SvR}, Stanley and van Roosmalen introduce \emph{narrow sequences} (later called \emph{ICE-sequences} in \cite{sakai} and in the present paper) of subcategories of $\A$ and show that they are in bijection with the set of \emph{homology-determined preaisles} in the bounded derived category $\derD(\A)$. This result generalizes the bijection between the set of torsion classes in $\A$ and the set of \emph{intermediate $t$-structures} in $\derD(\A)$ given by Happel-Reiten-Smal{\o} (HRS) tilting (see \cite{HRS,BR,woolf}). We focus in particular on the setting where $\A = \mods\Lambda$ is the category of finitely generated (right) modules over a finite-dimensional basic algebra $\Lambda$ over a field $K$ (from here referred to just as a ``finite-dimensional algebra''). In this setting, Sakai showed in \cite[Cor.~5.6]{sakai} that, for $m$ a positive integer, the bijection of Stanley and van Roosmalen restricts to a bijection between the set of \emph{contravariantly finite} ICE-sequences \emph{of length $m$} and the set of \emph{$(m+1)$-intermediate} $t$-structures whose aisles are homology-determined.
The goal of this paper is to classify contravariantly finite ICE-sequences using concepts from (the dual version of) the \emph{$\tau$-tilting theory} of \cite{AIR} (see also \cite{DF}).

The subcategories comprising an ICE-sequence are \emph{ICE-closed subcategories}, meaning that they are closed under images, cokernels, and extensions (see Definition~\ref{def:special_subcats}). ICE-closed subcategories were introduced in \cite{enomoto_rigid} as a simultaneous generalization of \emph{torsion classes} and \emph{wide subcategories}, each of which feature heavily in the study of finite-dimensional algebras (see e.g. \cite[Sec.~1.2]{BDH} or \cite[Sec.~1.1]{ES} for lists of references). In fact, while there are many examples of ICE-closed subcategories which are neither torsion classes nor wide subcategories, iterated versions of these definitions capture all ICE-closed subcategories. More precisely, wide subcategories are themselves abelian length categories, and thus come equipped with their own notion of a torsion class. It is shown in \cite[Thm.~A]{ES} that a subcategory $\C$ is ICE-closed if and only if there exists a wide subcategory $\W$ such that $\C$ is a torsion class in $\W$.

Related to the above is the study of ICE-closed subcategories via certain intervals of the lattice of torsion classes $\tors\A$, and the dual lattice of torsion-free classes $\torf\A,$ of $\A$. (The lattice theory of $\tors\A$ and $\torf\A$ has been the subject of recent intense study, see e.g. \cite[Sec.~1.2]{BDH} for a list of references.) To each interval $[\F,\G]$ in $\torf \A$, one associates the \emph{heart} $\G \cap {}^\perp \F$, see Notation~\ref{nota:perps} and Definition~\ref{def:wide_int}. Then $[\F,\G]$ is called a \emph{wide interval} (resp. \emph{ICE-interval}) if its heart is a wide subcategory (resp. ICE-closed subcategory). By \cite[Thm.~1.6]{AP} and \cite[Thm.~B]{ES}, both wide intervals and ICE-intervals can be characterized in purely lattice-theoretic terms. These characterizations are related to the so-called \emph{pop-stack sorting operators} (see Remark~\ref{rem:nuclear}), which appear in the context of dynamical algebraic combinatorics.

Now, by \cite[Thm.~5.12]{sakai}, one obtains a bijection between certain sequences of wide intervals in $\torf \A$ (the ``decreasing sequences of maximal join intervals'') and the set of (not necessarily contravariantly finite) ICE-sequences of finite length, see Section~\ref{sec:wide} for details. For $\A = \mods\Lambda$, we show in Lemma~\ref{lem:meet_to_ICE_ff} that the contravariantly finite ICE-sequences (of finite length) correspond to those sequences of intervals whose maximal and minimal elements are both functorially finite. Thus the problem of classifying contravariantly finite ICE-sequences of finite length reduces to the problem of classifying decreasing sequences of functorially finite maximal join intervals of torsion-free classes.

As previously mentioned, our goal is to understand contravariantly finite ICE-sequences using the $\tau$-tilting theory of \cite{AIR} (see also \cite{DF}). As we recall in Section~\ref{subsec:rigid}, functorially finite torsion and torsion-free classes in $\mods\Lambda$ can be classified using \emph{$\tau^{-1}$-rigid modules}.
By \cite{jasso,DIRRT}, one can also associate a functorially finite wide subcategory $\Jinv(X)$ to each $\tau^{-1}$-rigid module $X$ called the \emph{$\tau^{-1}$-perpendicular category}, see Definition-Proposition~\ref{defprop:jasso}. Moreover, it follows from \cite[Prop.~8.4]{BaH} that $\Jinv(X)$ arises as the heart of a (functorially finite) maximal join interval precisely when $X$ satisfies a technical condition called \emph{cogen-minimality}, see Definition~\ref{def:cogen} and Proposition~\ref{prop:meet}. The first main contribution of this paper uses the technique of \emph{$\tau$-tilting reduction} \cite{jasso,BM_exceptional} to extend this into a classification of decreasing sequences of functorially finite maximal join intervals. More preicsely, we introduce \emph{cogen-preordered} $\tau^{-1}$-rigid modules (Definition~\ref{def:cogen}) as sequences $(X_1,\ldots,X_m)$ of modules for which $\bigoplus_{k = 1}^m X_k$ is $\tau^{-1}$-rigid and which satisfy a technical conditon related to cogen-minimality. In case each $X_k$ is indecomposable, this coincides with the dual notion of a \emph{TF-ordered $\tau$-rigid module} from \cite{MT}. In that paper, TF-ordered $\tau$-rigid modules are shown to induce \emph{stratifying systems} and to be in bijection with the \emph{$\tau$-exceptional sequences} of \cite{BM_exceptional}. In the present paper, the more general cogen-preordered $\tau^{-1}$-rigid modules are shown to capture precisely the decreasing sequences of functorially finite maximal join intervals of torsion-free classes:

\begin{theorem}[see Theorem~\ref{thm:meet_and_TF} for details]\label{thm:mainA}    Let $\Lambda$ be a finite-dimensional algebra and let $m$ be a nonnegative integer. Then there is a bijective correspondence between
    \begin{itemize}
        \item[(1)] The set of isomorphism classes of cogen-preordered basic $\tau^{-1}$-rigid modules of length $m$ in $\mods\Lambda$, and
        \item[(2)] The set of decreasing sequences of functorially finite maximal join intervals of length $m$ in $\torf(\mods\Lambda)$.
    \end{itemize}
\end{theorem}

We note that the forward and reverse bijections in Theorem~\ref{thm:mainA} are constructed explicitly in Propositions~\ref{prop:TF_to_meet} and~\ref{prop:meet_to_TF}.

We then combine Theorem~\ref{thm:mainA} with the bijection in \cite[Thm.~5.12]{sakai} to prove the following.

\begin{corollary}[see Corollaries~\ref{cor:TF_to_ICE} and~\ref{cor:ICE_to_TF} for details]\label{cor:mainB}
    Let $\Lambda$ be a finite-dimensional algebra, and let $m$ be a nonnegative integer. Then there are bijective correspondences between
    \begin{itemize}
        \item [(1)] The set of isomorphism classes of cogen-preordered basic $\tau^{-1}$-rigid modules of length $m$ in $\mods\Lambda$, and
        \item [(2)] the set of contravariantly finite ICE-sequences of length $m+1$ in $\mods\Lambda$.
    \end{itemize}
\end{corollary}

Again, we note that both the forward and reverse bijections in Corollary~\ref{cor:mainB} are constructed explicitly.

Finally, we combine Corollary~\ref{cor:mainB} with \cite[Thm.~5.5]{sakai} to obtain our final main result.

\begin{corollary}\label{cor:mainC}
    Let $\Lambda$ be a finite-dimensional algebra, and let $m$ be a nonnegative integer. Then there are bijective correspondences between
    \begin{itemize}
        \item [(1)] The set of isomorphism classes of cogen-preordered basic $\tau^{-1}$-rigid modules of length $m$ in $\mods\Lambda$, and
        \item [(2)] the set of $(m+1)$-intermediate $t$-structures in the bounded derived category $\derD(\mods\Lambda)$ whose aisles are homology-determined.
    \end{itemize}
\end{corollary}

\subsection{Organization}

In Section~\ref{sec:subcats}, we recall background information about torsion classes, wide subcategories, and ICE-sequences. We then recall background information about wide intervals and their relationship with ICE-sequences in Section~\ref{sec:wide}. Section~\ref{sec:rigid} covers background information about $\tau^{-1}$-rigid modules and their perpendicular categories. In Section~\ref{sec:jasso}, we introduce cogen-preordered $\tau^{-1}$-rigid modules, establish some of their basic properties, and explain how they are related to (the dual version of) $\tau$-tilting reduction. Finally, we prove our main results in Section~\ref{sec:intervals}.


\section{Torsion pairs, wide subcategories, and ICE-sequences}\label{sec:subcats}

In this section, we recall background information about torsion classes, wide subcategories, and ICE-sequences.

\subsection{Conventions and definitions}

Let $\Lambda$ be a finite-dimensional algebra and denote $\M := \mods\Lambda$. Let $\A$ denote an essentially small abelian length category. We will mostly be interested in the case where $\A = \M$ or where $\A$ is an abelian subcategory of $\M$, but we introduce the additional notation to make the hypotheses of each statement more clear. Unless otherwise stated, every named object of $\A$ is assumed to be basic.

We typically consider objects of $\A$ only up to isomorphism. In particular, the word ``unique'' is used to mean ``unique up to isomorphism''. By a subcategory of $\A$, we will always mean a full additive subcategory closed under isomorphisms.
Given an object $X \in \A$, we denote by $\add X$ the \emph{additive closure} of $X$. Explicitly, $\add X$ is the subcategory whose objects are the direct summands of the objects $X^{\oplus k}$ for $k \in \mathbb{N}$.

\begin{notation}\label{nota:perps}
Given a subcategory $\C \subseteq \A$, we define several other subcategories as follows.
\begin{itemize}
	\item $\C^{\perp} := \{X \mid \Hom(-,X)|_{\C} = 0\}$.
	\item ${}^{\perp} \C := \{X\mid \Hom(X,-)|_{\C} = 0\}$.
	\item $\Cogen(\C) := \{X\mid \exists \text{ a monomorphism } X \hookrightarrow Y \text{ with } Y \in \C\}$.
	\item $\Gen(\C) := \{X\mid \exists \text{ an epimorphism } Y \twoheadrightarrow X \text{ with } Y \in \C\}$.
\end{itemize}
\end{notation}
We also apply each of these operators to an object $X \in \A$ by taking $\C = \add(X)$. The following observation will be useful later.

\begin{remark}\label{rem:cogen}
	Let $X, Y \in \A$. If $Y \notin \Cogen X$, then $\Cogen X \subsetneq \Cogen(X\oplus Y)$. Likewise if $Y \in \Cogen X$, then $\Cogen X = \Cogen (X\oplus Y)$.
\end{remark}

We will also need the following.

\begin{notation}\label{nota:star}
	Let $\B, \C \subseteq \A$ be subcategories. We denote by $\B * \C$ the subcategory consisting of those objects $X$ for which there exists a short exact sequence $0 \rightarrow B \rightarrow X \rightarrow C \rightarrow 0$ with $B \in \B$ and $\C \in \C$.
\end{notation}

The following types of subcategories will feature prominently in this paper.

\begin{definition}\label{def:special_subcats}
	A subcategory $\C \subseteq \A$ is said to be
	\begin{enumerate}
		\item \emph{closed under extensions} if $\C*\C \subseteq \C$.
		\item  \emph{closed under quotients} (resp. \emph{subobjects}) if every epimorphism $Y \twoheadrightarrow Z$ (resp. every monomorphism $X \hookrightarrow Y$) in $\A$ with $Y \in \C$ has $Z \in \C$ (resp. $X \in \C$).
		\item \emph{closed under images} (resp. \emph{kernels, cokernels}) if every morphism $f: X \rightarrow Y$ in $\A$ with $X, Y \in \C$ has $\im f \in \C$ (resp. $\ker f \in \C$, $\coker f \in \C$).
		\item a \emph{wide subcategory} if it is closed under extensions, kernels, and cokernels.
		\item a \emph{torsion class} if it is closed under extensions and quotients.
		\item a \emph{torsion-free class} if it is closed under extensions and subobjects.
		\item an \emph{ICE-closed subcategory} if it is closed under images, cokernels, and extensions.
	\end{enumerate}
\end{definition}

We denote by $\wide \A, \tors\A, \torf\A$, and $\ice \A$ the sets of wide subcategories, torsion classes, torsion-free classes, and ICE-closed subcategories of $\A$, respectively. Each of these is considered as a partially ordered set under the inclusion order. Each of these has the whole category $\A$ as its unique maximal element and the subcategory 0 (consisting of only the zero object) as its unique minimal element.


\subsection{Functorial finiteness}\label{subsec:ff}

Fix $\C \subseteq \A$ a subcategory $X \in \A$ an object. A \emph{left $\C$-approximation} of $X$ is a morphism $h^\C: X \rightarrow X^\C$ with $X^\C \in \C$ such that the map $\Hom(h^\C,C)$ is surjective for all $C \in \C$. Given a chain of subcategories $\C \subseteq \B \subseteq\A$, we say that $\C$ is \emph{contravariantly finite in $\B$} if every object of $\mathcal{B}$ admits a left $\C$-approximation. In case $\mathcal{B} = \A$, we just say that $\C$ is \emph{contravariantly finite}. The notions of \emph{right $\C$-approximations} and \emph{covariant finiteness} (in $\mathcal{B}$) are defined dually. A subcategory $\mathcal{C}$ which is both contravariantly finite and covariantly finite (in $\mathcal{B}$) is said to be \emph{functorially finite} (in $\mathcal{B}$). The following straightforward lemma will be useful.

\begin{lemma}\label{lem:chain_ff}
	Let $\mathcal{C} \subseteq \mathcal{B} \subseteq \A$ be a chain of subcategories and suppose that $\mathcal{B}$ is contravariantly finite (resp. covariantly finite, functorially finite). Then $\mathcal{C}$ is contravariantly finite (resp. covariantly finite, functorially finite) in $\mathcal{B}$ if and only if $\mathcal{C}$ is contravariantly finite (resp. covariantly finite, functorially finite).
\end{lemma}

We denote by $\fwide \A$ the subposet of $\wide \A$ consisting of the functorially finite wide subcategories. We denote $\ftors\A$ and $\ftorf\A$ analogously.

\subsection{Wide subcategories}\label{subsec:wide}

Fix a wide subcategory $\W \in \wide \A$ and a subcategory $\C \subseteq \A$. Then $\W$ is itself an abelian length category. Thus one can apply the constructions in Notation~\ref{nota:perps} to $\C$ considered either as a subcategory of $\W$ or considered as a subcategory of $\A$. To avoid notational ambiguity, our convention is that these operators are also applied in the category $\A$. We thus use intersections to apply them in $\W$, for example $$\C^\perp \cap \W = \{X \in \W \mid \Hom(-,X)|_{\C} = 0\}.$$
One can then talk about the wide subcategories, torsion classes, etc. of $\W$.

In case $\A = \M$, the functorial finiteness of $\W$ has the following characterization, see \cite[Prop.~4.12]{enomoto_rigid} for a proof.

\begin{lemma}\label{lem:ff_wide}
	A wide subcategory $\W \in \wide \M$ is functorially finite if and only if there exists a finite-dimensional algebra $\Lambda_\W$ admitting an exact equivalence $\W \simeq \mods\Lambda_\W$.
\end{lemma}

Lemma~\ref{lem:ff_wide} allows one to freely apply results about the category $\M$ to a functorially finite wide subcategory $\W$. We do this throughout the paper.

In \cite{IT,MS}, the authors consider a map $\WL: \tors\A \rightarrow \wide \A$. It is proved in \cite[Prop.~4.2]{enomoto_rigid} (see also \cite[Prop.~2.2]{sakai}) that one can extend the domain of $\WL$ to $\ice\A$. This map is defined as follows.

\begin{definition-proposition}\label{prop:ingalls_thomas}
	For $\C \in \ice \A$, denote
	$$\WL(\C) := \{X \in \C \mid (Y \in \C, f: Y \rightarrow X) \implies \ker f \in \C\}.$$
	Then $\WL(\C)$ is a wide subcategory.
\end{definition-proposition}

The following extension of Definition-Proposition~\ref{prop:ingalls_thomas} will also be useful.

\begin{proposition}\label{prop:ff_IT}\cite[Lem.~3.8]{MS} \textnormal{(see also \cite[Prop.~4.9]{ES})}
	Let $\W \in \fwide \M$ and $\T \in \ftors\W$. Then $\WL(\T) \in \fwide \M$.
\end{proposition}


\subsection{Torsion pairs}\label{subsec:tors}

We now recall further details about torsion classes and torsion-free classes. We refer to \cite[Sec.~VI.1]{ASS} and \cite{thomas_intro} for additional background.

A pair of subcategories $(\T,\F)$ (of $\A$) is called a \emph{torsion pair} (in $\A$) if $\T = {}^{\perp} \F$ and $\F = \T^{\perp}$. It is well-known that a subcategory $\T$ is a torsion class if and only if $(\T,\T^\perp)$ is a torsion pair. Dually, a subcategory $\F$ is a torsion-free class if and only if $({}^\perp \F,\F)$ is a torsion pair. More generally, if $\W \in \wide \A$ is a wide subcategory and $\F \in \torf \A$ is a torsion-free class of $\W$, then $(\W \cap {}^\perp \F,\F)$ is a torsion pair of $\W$.
 
Each torsion pair $(\T,\F)$ comes equipped with a pair of additive functors $t: \A \rightarrow \T$ and $f: \A \rightarrow \F$ called the \emph{torsion functor} and \emph{torsion-free functor}, respectively. Each object $X$ is then the middle term of a \emph{canonical short exact sequence} 
$$0 \rightarrow tX \rightarrow X \rightarrow fX \rightarrow 0.$$
The maps comprising this sequence are a right $\T$-approximation of $X$ and a left $\F$-approximation of $X$, respectively. In particular, torsion classes are always covariantly finite and torsion-free classes are always contravariantly finite.

In this case $\A = \mods\Lambda$, we also have the following.

\begin{lemma}\cite[Thm.]{smalo}\label{lem:torsion_pair_ff}
	Let $(\T,\F)$ be a torsion pair in $\M$. Then the following are equivalent.
	\begin{enumerate}
		\item $\T$ is contravariantly finite.
		\item $\T$ is functorially finite.
		\item $\F$ is covariantly finite.
		\item $\F$ is contravariantly finite.
	\end{enumerate}
\end{lemma}

The sets $\tors\A$ and $\torf\A$ are both (complete) lattices under the inclusion order \cite{IRTT}. This means every set $\mathfrak{F} \subseteq \torf\A$ has a unique \emph{join} (least upper bound) $\bigvee \mathfrak{F}$ and a unique \emph{meet} (greatest lower bound) $\bigwedge \mathfrak{F}$, and likewise for $\mathfrak{F} \subseteq \tors\A$. Explicit formulas for $\bigvee \mathfrak{F}$ and $\bigwedge \mathfrak{F}$ will not be needed in this paper.

The lattices $\tors\A$ and $\torf\A$ are related by anti-isomorphisms $(-)^{\perp}: \tors\A \rightarrow \torf\A$ and ${}^{\perp}(-): \torf\A \rightarrow \tors\A$. We will work mostly with the lattice $\torf \A$ in this paper, which can also be identified with the lattice of torsion classes of the opposite category $\A^{\mathrm{op}}$. In particular, this allows us to apply several results from the literature stated for $\tors\A$ to the lattice $\torf \A$.


\subsection{ICE-sequences}\label{subsec:ICE}
   We recall the definition of an \emph{ICE-sequence} from \cite{sakai}. As mentioned in the introduction, ICE-sequences coincide with the ``narrow sequences'' of \cite{SvR} by \cite[Prop.~4.2]{sakai}.

\begin{definition}\cite[Sec.~4 and~5]{sakai}\label{def:ICE}
    A sequence $(\C(k))_{k \in \mathbb{Z}}$ is said to
    \begin{enumerate}
        \item be an \emph{ICE-sequence} if, for any $k \in \mathbb{Z}$, $\C(k) \in \ice \A$ and $\C(k+1) \in \tors \WL(\C(k))$.
        \item be \emph{contravariantly finite} if every $\C(k)$ is contravariantly finite.
        \item have \emph{length $m$} (for $m$ a nonnegative integer) if $\C(1) = 0$ and $\C(-m) = \A$.
    \end{enumerate}
\end{definition}

\begin{remark}\label{rem:ICE_stable}
    Note that if an ICE-sequence has length $m$ and only if $\C(k) = 0$ for all $k > 0$ and $\C(k) = \A$ for all $k \leq -m$. In particular an ICE sequence of length $m$ also has length $m'$ for all $m' > m$.
\end{remark}

In case $\A = \M$, we have the following characterization of contravariantly finite ICE-sequences.

\begin{lemma}\label{lem:ff_ice}
    Suppose $\A = \M$ and let $(\C(k))_{k \in \mathbb{Z}}$ be an ICE-sequence of length $m$. Then $(\C(k))_{k \in \mathbb{Z}}$ is contravariantly finite if and only if $\C(k) \in \ftors \WL(\C(k-1))$ for every $-m < k \leq 0$. Moreover, if $(\C(k))_{k \in \mathbb{Z}}$ is contravariantly finite, then $\WL(\C(k))$ and $\C(k)$ are both functorially finite for all $k$.
\end{lemma}

\begin{proof}
    Let $(\C(k))_{k \in \mathbb{Z}}$ be an ICE-sequence of length $m$. Suppose first that $(\C(k))_{k \in \mathbb{Z}}$ is contravariantly finite. By definition, this means each $\C(k)$ is contravarianty finite, and thus is also contravariantly finite in $\WL(\C(k-1))$. Since $\C(k) \in \tors(\WL(\C(k-1)))$, it follows that $\C(k)$ is also covariantly finite in $\WL(\C(k-1))$, and thus that $\C(k) \in \ftors(\WL(\C(k-1))$.

    Now suppose that $\C(k) \in \ftors \WL(\C(k-1))$ for every $-m < k \leq 0$. Then, by Lemma~\ref{lem:chain_ff} and Remark~\ref{rem:ICE_stable}, it suffices to show that $\WL(\C(k))$ is functorially finite for all $-m \leq k \leq 0$. We prove this by induction. The base case $k = -m$ follows from the fact that $\C(-m) = \M$, so suppose that $k > -m$ and that $\WL(\C(k-1)) \in \fwide \M$. Then $\WL(\C(k))$ is also functorially finite by Proposition~\ref{prop:ff_IT}.
\end{proof}


\section{Wide intervals and maximal join intervals}\label{sec:wide}

Let $\sL \in \{\tors\A,\torf\A\}$. By an \emph{interval} in $\sL$, we will always mean a closed interval, i.e., a subset of the form $[\B,\C] = \{\mathcal{L} \in \sL \mid \B \subseteq \mathcal{L} \subseteq \C\}$ for some $\B \subseteq \C \in \sL$. For $\mathsf{I} \subseteq L$ an interval, we denote by $\sI^+$ and $\sI^-$ the maximum and minimum elements of $\sI$, respectively. We say that $\sI$ is \emph{functorially finite} if both $\sI^+$ and $\sI^-$ are functorially finite.

In this section, we study special types of intervals in detail. We first recall the definition of a wide interval and the associated reduction process in Section~\ref{subsec:wide_ints}. We then recall the definitions of maximal join and meet intervals in Section~\ref{subsec:join}. Finally, in Section~\ref{subsec:ICE_ints},
we recall the relationship between maximal join/meet intervals and ICE-sequences established in \cite{sakai}. We also explain how this relationship interacts with the notion of functorial finiteness.

\subsection{Wide intervals and reduction}\label{subsec:wide_ints}

We adopt the following terminology from \cite{AP,tattar}.

\begin{definition}\label{def:wide_int} Let $\sI$ be an interval in $\torf\A$ (resp. $\tors\A$).
			\begin{enumerate}
				\item The \emph{heart} of $\sI$ is the subcategory $\H_\sI := \sI^+ \cap {}^\perp (\sI^-)$ (resp. $\H_\sI := (\sI^-)^\perp \cap \sI^+$  ).
				\item $\sI$ is a \emph{wide interval} if its heart $\H_\sI$ is a wide subcategory.
			\end{enumerate}
\end{definition}

\begin{remark}\label{rem:hearts}
	Let $\sI$ be an interval in $\torf\A$. Then $\H_\sI = \H_{[{}^\perp (\sI^+),{}^\perp (\sI^-)]}$. In particular, $\sI$ is a wide interval in $\torf\A$ if and only if $[{}^\perp (\sI^+),{}^\perp (\sI^-)]$ is a wide interval in $\tors\A$. Moreover, if $\A = \M$, then $\sI$ is a functorially finite wide interval in $\torf\M$ if and only if $[{}^\perp(\sI^+),{}^\perp(\sI^-)]$ is a functorially finite wide interval in $\tors\M$ by Lemma~\ref{lem:torsion_pair_ff}.
\end{remark}

The following result will be useful.

\begin{lemma}\label{lem:ff_ints}\cite[Lem.~4.18]{ES}
    Let $\sI \subseteq \torf\M$ be an interval. Then:
    \begin{enumerate}
        \item If any two of $\sI^+, \sI^-$, and $\H_\sI$ are functorially finite, then so is the third.
        \item If $\sI$ is a wide interval and either of $\sI^+$ or $\sI^-$ is functorially finite, then so is the other.
    \end{enumerate}
\end{lemma}

For $\sI \subseteq \torf \A$ a wide interval, we have the following relationship between the interval $\sI$ and the lattice $\torf \H_\sI$. 
Recall the notation $\B * \C$ from Notation~\ref{nota:star}.

\begin{proposition}\cite[Thm.~5.2]{AP}\label{prop:AP}
	Let $\sI \subseteq \torf\A$ be a wide interval. Then the associations $\F \mapsto \H_{[\sI^-,\F]} = \F \cap {}^\perp (\sI^-) = \F \cap \H_{\sI}$ and $\G \mapsto \G * \sI^-$ yield inverse order-preserving bijections between $\sI$ and $\torf \H_\sI$.
\end{proposition}

\begin{remark}\label{rem:AP_bijections}
    Special cases of Proposition~\ref{prop:AP} arising in the context of $\tau$-tilting theory were previously proved in \cite{jasso,DIRRT}. While these special cases constitute the main use of Proposition~\ref{prop:AP} in the present paper, we will also require the more general statement of \cite{AP} referenced above. Note also that Proposition~\ref{prop:AP} is further generalized in \cite[Thm.~A]{tattar}, which plays a role in the proof of Lemma~\ref{lem:inverse_tattar} below.
\end{remark}

For use later, we also record the following technical lemma.

\begin{lemma}\label{lem:inverse_tattar}
    Let $\sI \subseteq \torf \A$ be a wide interval and let $\F \in \sI$. Then
    $$\F = \left(\H_{[\F,\sI^+]}^\perp \cap \H_{\sI}\right)* \sI^-.$$
\end{lemma}

\begin{proof}
    Recall that we have torsion pairs $({}^\perp \F,\F)$, $({}^\perp (\sI^-),\sI^-)$ and $({}^\perp (\sI^+),\sI^+)$ in $\A$. These satisfy ${}^\perp(\sI^+) \subseteq {}^\perp \F \subseteq {}^\perp(\sI^-)$. By \cite[Thm.~A]{tattar}, it follows that  $(\H_{[\F,\sI^+]}, \H_{[\sI^-,\F]})$ is a torsion pair in $\H_{\sI}$. Thus $\H_{[\F,\sI^+]}^\perp \cap \H_\sI = \H_{[\sI^-,\F]}$. The result then follows from Proposition~\ref{prop:AP}.
\end{proof}

\subsection{Maximal join intervals}\label{subsec:join}

Let $\sL \in \{\tors\A,\torf\A\}$. For $\C \in \sL$,  we denote
\begin{align*}
	\mathrm{cov}^\uparrow(\C) &= \{\B \in \sL \mid \C \subsetneq \B \text{ and } ((\C \subsetneq \B' \subseteq \B) \implies (\B' = \B))\},\\
	\mathrm{cov}_\downarrow(\C) &= \{\B \in \sL \mid \B \subsetneq \C \text{ and } ((\B \subseteq \B' \subsetneq \C) \implies (\B' = \B))\}.
\end{align*}
For an interval $\sI\subseteq\sL$, we then denote
\begin{align*}
	\mathrm{pop}^{\sI^+}(\sI^-) &= \bigvee \{\sI^-\} \cup \{\C \in \cov^\uparrow(\sI^-) \mid \C \subseteq \sI^+\},\\
	\mathrm{pop}_{\sI^-}(\sI^+) &= \bigwedge \{\sI^+\} \cup \{\C \in \cov_\downarrow(\sI^+): \sI^- \subseteq \C\}.
\end{align*}

Before making a remark about the notation above, we recall the following.

\begin{proposition}\cite[Thm.~1.6]{AP}\label{prop:AP2}
	Let $\sL \in \{\tors\A,\torf\A\}$ and let $\sI$ be an interval in $\sL$. Then the following are equivalent.
	\begin{enumerate}
		\item $\sI$ is a wide interval.
		\item $\sI^- = \mathrm{pop}_{\sI^-}(\sI^+)$.
		\item There exists $\C \in \sL$ such that $\sI^- = \mathrm{pop}_{\C}(\sI^+)$.
		\item $\sI^+ = \mathrm{pop}^{\sI^+}(\sI^-)$.
		\item There exists $\C \in \sL$ such that $\sI^+ = \mathrm{pop}^{\C}(\sI^-)$.
	\end{enumerate}
\end{proposition}

\begin{remark}\label{rem:nuclear}
    The notation $\mathrm{pop}^{\sI^+}(\sI^-)$ and $\mathrm{pop}_{\sI^-}(\sI^+)$ comes from \cite[Sec.~2]{facial}, and is motivated by the connection between these constructions and the so-called ``pop-stack sorting operators''. We refer to the introduction of \cite{BDH}, and the references therein, for an explanation of the pop-stack sorting operators on the lattice $\tors\M$ and their combinatorial and representation-theoretic significance. We highlight in particular that intervals satisfying the equivalent conditions in Proposition~\ref{prop:AP2} have also been called \emph{binuclear intervals}, \emph{join intervals}, and \emph{meet intervals} in the literature, see e.g. \cite{AP,facial}.
\end{remark}

We also consider the following variants of \cite[Definition~5.10]{sakai}. Note that two intervals satisfy $\mathsf{J} \subseteq \sI$ if and only if $\sI^- \subseteq \mathsf{J}^- \subseteq \mathsf{J}^+ \subseteq \sI^+$.

\begin{definition}\label{def:maximal_int}
	Let $\sL \in \{\tors\A,\torf\A\}$ and let $\sI$ be an interval in $\sL$.
	\begin{enumerate}
		\item $\sI$ is a \emph{maximal join interval in $\sL$} (resp. \emph{maximal meet interval in $\sL$}) if $\sI^+ = \mathrm{pop}^{\A}(\sI^-)$ (resp. $\sI^- = \mathrm{pop}_0(\sI^+)$).
		\item An interval $\mathsf{J} \subseteq \sI$ of $\sL$ is a \emph{maximal join interval in $\sI$} (resp. \emph{maximal meet interval} in $\sI$) if $\mathsf{J}^+ = \mathrm{pop}^{\sI^+}(\mathsf{J}^-)$ (resp. $\mathsf{J}^- = \mathrm{pop}_{\sI^-}(\mathsf{J}^+)$).
	\end{enumerate}
\end{definition}

\begin{remark}\label{rem:meet_join}
	Let $[\F,\G]$ be an interval in $\torf \A$. Then $[\F,\G]$ is a maximal join (resp. meet) interval in $\torf\A$ if and only if $[{}^\perp\G,{}^\perp\F]$ is a maximal meet (resp. join) interval in $\tors\A$. More generally, $[\F,\G]$ is a maximal join (resp. meet) interval in some $[\B,\C]$ if and only if $[{}^\perp\G,{}^\perp\F]$ is a maximal meet (resp. join) interval in $[{}^\perp\C,{}^\perp\B]$.
\end{remark}

We are now ready to state the following definition.

\begin{definition}\label{def:decreasing}
	Let $m$ be a nonnegative integer, let $\sL \in \{\tors\A,\torf\A\}$, and let $\mathfrak{I} = (\sI_1,\ldots,\sI_m)$ be a sequence of intervals in $\sL$. We set $\sI_{m+1} = [0,\A]$. We say that $\mathfrak{I}$ is a \emph{decreasing sequence of (functorially finite) maximal join} (resp. \emph{meet}) \emph{intervals} (of length $m$) in $\sL$ if, for all $1 \leq k \leq m$, we have that $\sI_k$ is a (functorially finite) maximal join (resp. meet) interval in $\sI_{k+1}$.
\end{definition}

\begin{remark}
    Note that the indexing of Definition~\ref{def:decreasing} differs from that in \cite[Sec.~5]{sakai}. This deviation is made in order to simplify the formulas for the bijections in Theorem~\ref{thm:meet_and_TF}.
\end{remark}


\subsection{ICE-sequences via maximal join intervals}\label{subsec:ICE_ints}

The following combines Remarks~\ref{rem:hearts} and~\ref{rem:meet_join} with \cite[Thm.~5.12]{sakai}.

\begin{proposition}\label{prop:meet_to_ICE}
    Let $m$ be a nonnegative integer. Then there is a bijection $\nu$ from the set of decreasing sequences of maximal join intervals of length $m$ in $\torf \A$ to the set of ICE-sequences of length $m+1$ in $\A$ given as follows. Let $\mathfrak{J} = (\sI_1,\ldots,\sI_m)$ be a decreasing sequence of maximal join intervals and denote $\sI_{m+1} = [0,\A]$. For $k \in \mathbb{Z}$, denote
    $$\C(k) = \begin{cases} 0 & \text{if $k > 0$}\\ \H_{[\sI_{1-k}^-,\sI_{2-k}^+]}&\text{if $-m < k \leq 0$}\\\A &\text{if $k \leq -m$}.\end{cases}$$
    Then $\nu(\mathfrak{J}) = (\C(k))_{k \in \mathbb{Z}}$. Moreover, we have $\WL(\C(k)) = \H_{\sI_{1-k}}$ for $-m \leq k \leq 0$.
\end{proposition}

\begin{lemma}\label{lem:meet_to_ICE_ff}
    Let $\mathfrak{J} = (\sI_1,\ldots,\sI_m)$ be a decreasing sequence of maximal join intervals in $\torf \M$. Then $\nu_\M(\mathfrak{J})$ is contravariantly finite if and only if $\mathfrak{J}$ is functorially finite.
\end{lemma}

\begin{proof}
    Suppose first that $\mathfrak{J}$ is functorially finite. Then, for $-m < k \leq 0$, the interval $[\sI_{1-k}^-,\sI_{2-k}^+]$ is functorially finite. Thus $\C(k) = \H_{[\sI_{1-k}^-,\sI_{2-k}^+]}$ is functorially finite by Lemma~\ref{lem:ff_ints}. We conclude that $\nu(\mathfrak{J})$ is contravariantly finite.

    Now suppose that $\nu(\mathfrak{J})$ is contravariantly finite. We show that each $\sI_k$ is functorially finite by reverse induction on $k$. The base case $k = m+1$ holds by the convention that $\sI_{m+1} = [0,\M]$. Thus suppose that $k \leq m$ and that $\sI_{k+1}$ is functorially finite. Then in particular $\sI_{k+1}^+$ is functorially finite. Moreover, $\C(1-k) = \H_{[\sI_k^-,\sI_{k+1}^+]}$ is also functorially finite by Lemma~\ref{lem:ff_ice}. Thus $\sI_k^-$ and $\sI_k^+$ are both functorially finite by Lemma~\ref{lem:ff_ints}.
\end{proof}

In lieu of computing the full inverse of $\nu$, we will make use of the following.

\begin{proposition}\label{prop:ICE_to_meet}
    Let $(\C(k))_{k \in \mathbb{Z}}$ be an ICE-sequence of length $m+1$ in $\A$. Write $\nu^{-1}((\C(k))_{k \in \mathbb{Z}}) = (\sI_1,\ldots,\sI_m)$ and, for $1 \leq k \leq m+1$, reverse-inductively denote
    $$\F_k = \begin{cases} 0 & k = m+1\\\left[(\C(1-k)^\perp \cap \WL(\C(-k))\right] * \F(k+1) & 1 \leq k \leq m.\end{cases}$$
    Then $\sI_k^- = \F_k$ for all $1 \leq k \leq m+1$.
\end{proposition}

\begin{proof}
    We prove the result by reverse induction on $k$. The base case $k = m+1$ follows from the convention that $\sI_{m+1} = [0,\A]$. Thus suppose that $k \leq m$ and that $\sI_{k+1}^- = \F_{k+1}$. Now $\C(1-k) = \H_{[\sI_k^-,\sI_{k+1}^+]}$ and $\WL(\C(-k)) = \H_{\sI_{k+1}}$ by Proposition~\ref{prop:meet_to_ICE}. Since $\sI_k^- \in \sI_{k+1}$, Lemma~\ref{lem:inverse_tattar} then implies that $\sI^-_k = \F_k$.
\end{proof}


\section{$\tau^{-1}$-rigid modules and $\tau^{-1}$-perpendicular categories}\label{sec:rigid}

In this section, we  recall background information about \emph{$\tau^{-1}$-rigid modules} and their connection to functorially finite torsion pairs and wide intervals. We work exclusively in the category $\A = \M$. We denote by $\tau_\M$ and $\tau^{-1}_\M$ the Auslander-Reiten translate and its dual. 


\subsection{$\tau^{-1}$-rigid modules}\label{subsec:rigid}

We recall the following.

\begin{definition}\cite[Sec.~2.2]{AIR}\label{def:trig} A module $X \in \M$ is said to be \emph{$\tau^{-1}_\M$-rigid} if $\Hom\left(\tau^{-1}_\M X,X\right) = 0$.
\end{definition}

Functorially finite torsion pairs are related to $\tau^{-1}_\M$-rigid modules by the following.

\begin{proposition}\cite[Sec.~2.3]{AIR}\cite{AS}\label{prop:ff_torsion}
    Let $X$ be a $\tau^{-1}_\M$-rigid module. Then there are functorially finite torsion pairs
    $$(\Gen \tau^{-1}_\M X, \tau^{-1}_\M X^\perp) \qquad \text{and} \qquad ({}^\perp X, \Cogen X).$$
    Moreover, the association $X \mapsto \Cogen X$ is a surjection from the set of (basic) $\tau^{-1}_\M$-rigid modules to the set $\ftorf \M$.
\end{proposition}

\begin{remark}
    Since our convention is to consider objects only up to isomorphism, we use the terminology ``the set of $\tau^{-1}_\M$-rigid modules'' in Proposition~\ref{prop:ff_torsion} for what is more precisely the set of isomorphism classes of $\tau^{-1}_\M$-rigid modules. We adopt similar terminology throughout the paper. Note also that the word ``basic'' is included in Proposition~\ref{prop:ff_torsion} only to remind the reader that all named modules are assumed to be basic throughout this paper.
\end{remark}

We introduce some additional terminology in order to define sections of the surjection in Proposition~\ref{prop:ff_torsion}. For a module $X \in \M$ and a subcategory $\C \subseteq \M$, we denote by $X/\C$ the module obtained by deleting all direct summands of $X$ which lie in $\C$. For example, if $X, Y$, and $Z$ are all indecomposable, then $(X\oplus Y)/\add(Y\oplus Z) = X$. We then recall the following definition.

\begin{definition}\label{def:cogen}
	A module $X$ is \emph{cogen-minimal} if every indecomposable direct summand $Y$ of $X$ satisfies $Y \notin \Cogen(X/\add(Y))$.
\end{definition}

Now fix a subcategory $\C \subseteq \mods\Lambda$ which is closed under extensions. A module $X \in \C$ is said to be \emph{Ext-injective} in $\C$ if $\Ext^1(-,X)|_\C = 0$ and is said to be \emph{split injective} in $\C$ if every monomorphism $X \hookrightarrow Y$ with $Y \in \C$ is split. We denote by $\I(\C)$ and $\Is(\C)$ the subcategories consisting of the Ext-injective and split injective objects in $\C$, respectively. When they exist, we denote by $\Imod(\C)$ and $\Imods(\C)$ the minimal additive generators of $\I(\C)$ and $\Is(\C)$; that is, $\Imod(\C)$ is the unique basic module such that $\add(\Imod(\C)) = \I(\C)$, and likewise for $\Imods(\C)$. It is straightforward to show that $\Is(\C) \subseteq \I(\C)$, and thus also that $\Imods(\C)$ is a direct summand of $\I(\C)$ when these generators exist. Note also that the module $\Imods(\C)$ is necessarily cogen-minimal (if it exists).

The above constructions yield the following.

\begin{proposition}\cite[Sec.~2.3]{AIR}\label{prop:cogen_min_ff}\
\begin{enumerate}
    \item The association $X \mapsto \Cogen X$ is a bijection from the set of cogen-minimal $\tau^{-1}_\M$-rigid modules to the set $\ftorf\M$. Its inverse is given by $\F \mapsto \Imods(\F)$.
    \item Let $\F \in \ftorf\M$. Then $\Imod(\F)$ is a $\tau^{-1}_\M$-rigid module which satisfies $\F = \Cogen(\Imod(\F))$.
\end{enumerate}
\end{proposition}

\begin{remark}
    The module $\Imod(\F)$ in Proposition~\ref{prop:cogen_min_ff}(2) is a \emph{support $\tau^{-1}_\M$-tilting module}, and one can also define a section of the map  in Proposition~\ref{prop:ff_torsion} by restricting to the support $\tau^{-1}_\M$-tilting modules. We will not, however, make explicit use of support $\tau^{-1}_\M$-tilting modules in this paper.
\end{remark}

The following is also useful. For the convenience of the reader, we give a short proof deducing this statement from the cited result. 

\begin{lemma}\cite[Prop.~2.9]{AIR}\label{lem:sum_tau_rigid}
    Let $X$ be a $\tau^{-1}_\M$-rigid module and let $\F \in \ftorf\M$ such that $\Cogen X \subseteq \F \subseteq (\tau_\M^{-1}X)^\perp$. Then $X \oplus ((\Imods(\F)/\add X))$ is a (basic) $\tau^{-1}_\M$-rigid module which satisfies $\Cogen(X \oplus ((\Imods(\F)/\add X))) = \F$.
\end{lemma}

\begin{proof}
    Suppose that $\Cogen X \subseteq \F \subseteq (\tau_\M^{-1}X)^\perp$. Then $X \in \I(\F)$ by the dual of \cite[Prop.~2.9]{AIR}. Thus we have a chain $\Imods(\F)\subseteq X \oplus ((\Imods(\F)/\add X))\subseteq \Imod(\F)$, where each inclusion is as a direct summand. The result then follows from Proposition~\ref{prop:cogen_min_ff} and Remark~\ref{rem:cogen}.
\end{proof}


\subsection{$\tau^{-1}$-perpendicular categories}\label{subsec:perpendicular}

Recall from Proposition~\ref{prop:ff_torsion} that every $\tau^{-1}_\M$-rigid module $X$ induces two torsion-free classes: $\Cogen X$ and $(\tau^{-1}_\M X)^\perp$. It follows immediately from the definition of $\tau^{-1}_\M$-rigid that $\Cogen X \subseteq (\tau^{-1}_\M X)^\perp$; i.e., that $[\Cogen X, (\tau^{-1}_\M X)^\perp]$ is an interval in $\torf \M$. Moreover, we have the following.

\begin{definition-proposition}\cite[Thm.~1.4]{jasso}\cite[Thm.~4.12]{DIRRT}\label{defprop:jasso}
    Let $X$ be a $\tau^{-1}_\M$-rigid module. Then $[\Cogen X, (\tau^{-1}_\M X)^\perp]$ is a functorially finite wide interval whose heart is the \emph{$\tau^{-1}$-perpendicular category} of $X$:
    $$\Jinv_\M(X) := \H_{[\Cogen X, (\tau^{-1}_\M X)^\perp]} = (\tau^{-1}_\M X)^\perp \cap {}^\perp X.$$
    Moreover, $\Jinv_\M(X)$ is exact equivalent to $\mods\Lambda_X$ for some finite-dimensional algebra~$\Lambda_X$.
\end{definition-proposition}

\begin{remark}
    Implicit in Definition-Proposition~\ref{defprop:jasso} is the fact that the $\tau^{-1}$-perpendicular category $\Jinv_\M(X)$ is functorially finite, see Lemma~\ref{lem:ff_wide}.
\end{remark}


\begin{notation}
As a consequence of Definition-Proposition~\ref{defprop:jasso}, one can replace the category $\M = \mods\Lambda$ with a $\tau^{-1}$-perpendicular subcategory $\W \subseteq \M$ in all of the constructions which have appeared in this section. For example, if $X \in \W$ is a $\tau^{-1}_\W$-rigid module, then its $\tau^{-1}$-perpendicular category is $$\Jinv_\W(X) = (\tau_\W^{-1}X)^\perp \cap ({}^\perp X) \cap \W.$$
\end{notation}



\section{Cogen-preorderings}\label{sec:jasso}

In this section, we introduce cogen-preorderings of $\tau^{-1}_\M$-rigid modules. These generalize (the dual versions of) TF-orderings of $\tau$-rigid modules from \cite{MT}. We then restate the (dual version of the) relationship between $\tau^{-1}_\M$-rigid modules and $\tau^{-1}_{\Jinv_\M(X)}$-rigid modules from \cite{jasso,BM_exceptional} using this new language.


\subsection{Cogen-preorderings}\label{subsec:cogen}

Let $\Delta = (X_1,X_2,\ldots,X_m)$ be a sequence of modules.  The \emph{length} of $\Delta$ is $m$. We allow for the case where $\Delta = ()$ is the empty sequence, which corresponds to the case $m = 0$.

For $0 \leq k \leq m$, we denote $\Delta_{> k} := \bigoplus_{j > k}X_j$ and denote $\Delta_{\geq k}$ analogously. We also denote $\bigoplus\Delta := \Delta_{> 0}$. If $\bigoplus\Delta$ is basic, then we refer to $\Delta$ as a \emph{preordered decomposition} of $\bigoplus \Delta$. If in addition each $X_i$ is indecomposable, we say that $\Delta$ is an \emph{ordered decomposition} of $\bigoplus \Delta$. In particular, this means that the empty sequence is an ordered decomposition of the zero module.

 For a fixed preordered decomposition $\Delta = (X_1,\ldots,X_m)$ of a (basic) module $X$ and an indecomposable direct summand $Y$ of $X$, there exists a unique index $\iota_\Delta(Y)$ such that $Y$ is a direct summand of $X_{\iota_\Delta(Y)}$. We say that an ordered decomposition $\Gamma = (Y_1,\ldots,Y_m)$ of $X$ is an \emph{ordering of $\Delta$} if $\iota_\Delta(Y_j) \leq \iota_\Delta(Y_k)$ for all $1 \leq j < k \leq m$.

 \begin{remark}
    For $\Delta$ a preordered decomposition of a module $X$, the relation $Y \leq Y' \iff \iota_\Delta(Y) \leq \iota_\Delta(Y')$ is a total preorder on the set of indecomposable direct summands of $X$. This is our justification for the name ``preordered decomposition''.
\end{remark}

We consider the following generalizations of (the dual of) \cite[Def.~3.1]{MT}. Note that the dual of Definition~\ref{def:TF}(4) coincides with the definition of a ``TF-admissible decomposition'' given in \cite[Def.~3.1]{MT}.

\begin{definition}\label{def:TF}Let $\Delta = (X_1,\ldots,X_m)$ be a sequence of modules.
\begin{enumerate}
    \item We say that $\Delta$ is a \emph{preordered} (resp. \emph{ordered}) \emph{$\tau^{-1}_\M$-rigid module} if $\Delta$ is a preordered (resp. ordered) decomposition of a (basic) $\tau^{-1}_\M$-rigid module $\bigoplus \Delta$.

    \item We say that $\Delta$ is a \emph{weakly cogen-preordered $\tau^{-1}_\M$-rigid module} if the following both hold.
	\begin{enumerate}
		\item $\Delta$ is a preordered  $\tau^{-1}$-rigid module.
		\item For every $1 \leq k \leq m$ and every indecomposable direct summand $Y$ or $X_k$, one has $Y \notin \Cogen\left((\Delta_{> k})/\add Y\right)$.
	\end{enumerate}
	\item Suppose that $\Delta$ is a weakly cogen-preordered $\tau^{-1}_\M$-rigid module. We say that $\Delta$ is a \emph{cogen-preordered $\tau^{-1}_\M$-rigid module} if (2b) above can be strengthened to
	\begin{itemize}
		\item[(b')] For every $1 \leq k \leq m$ and every indecomposable direct summand $Y$ or $X_k$, one has $Y \notin \Cogen\left((\Delta_{\geq k})/\add Y\right)$.
	\end{itemize}
	\item Suppose that $\Delta$ is a cogen-preordered $\tau^{-1}_\M$-rigid module. We say that $\Delta$ is a \emph{cogen-ordered $\tau^{-1}_\M$-rigid module} if $\Delta$ is an ordered decomposition (of $\bigoplus \Delta$). Equivalently, this means each $X_k$ is indecomposable.
	\end{enumerate}
\end{definition}

\begin{remark}
    In the setting of Definition~\ref{def:TF}, we have that $\Delta_{\geq k}/(\add X_k) = \Delta_{> k}$. Thus conditions (2b) and (3b') of Definition~\ref{def:TF} coincide in the case where each $X_k$ is indecomposable.
\end{remark}

 The following is also clear from the definitions.

\begin{lemma}\label{lem:suborder}
	Let $\Delta = (X_1,\ldots,X_m)$ be a preordered $\tau^{-1}_\M$-rigid module. Then $\Delta$ is cogen-preordered if and only if every indecomposable ordering of $\Delta$ is cogen-ordered.
\end{lemma}

The following characterization of cogen-preorderings will also be useful.

\begin{lemma}\label{lem:cogen_preorder}
	Let $\Delta = (X_1,\ldots,X_m)$ be a preordered $\tau^{-1}_\M$-rigid module. Then $\Delta$ is cogen-preordered if and only if $X_k = (\Imods(\Cogen \Delta_{\geq k}))/\Cogen \Delta_{> k}$ for all $1 \leq k \leq m$.
\end{lemma}

\begin{proof}
	Suppose first that $\Delta$ is not cogen-preordered. Thus there exists some index $k$ and a direct summand $Y$ of $X_k$ such that $Y \in \Cogen(\Delta_{\geq k}/\add Y)$. By Remark~\ref{rem:cogen}, we then have $\Cogen(\Delta_{\geq k}/\add Y) = \Cogen(\Delta_{\geq k})$. Thus $Y$ is not split injective in $\Cogen(\Delta_{\geq k})$, and so $X_k \neq (\Imods(\Cogen \Delta_{\geq k}))/(\Cogen \Delta_{> k})$.
	
	Conversely, suppose that there exists some $1 \leq k \leq m$ with $X_k \neq (\Imods(\Cogen \Delta_{\geq k}))/(\Cogen \Delta_{> k})$. The definitions of split injective and $\Cogen(-)$ imply that $\Imods(\Cogen \Delta_{\geq k})$ is a direct summand of $\Delta_{\geq k}$, see e.g. \cite[Cor.~1.10b]{BMH}. Thus there exists a direct summand $Y$ of $X_k$ such that $Y$ is not split injective in $\Cogen \Delta_{\geq k}$. By Remark~\ref{rem:cogen}, this means $\Cogen\Delta_{\geq k} = \Cogen (\Delta_{\geq k}/\add Y)$. Thus $Y \in \Cogen (\Delta_{\geq k}/\add Y)$, and so $\Delta$ is not cogen-preordered.
\end{proof}


\subsection{Reduction}\label{subsec:reduction}

For use throughout this section, we fix the following notation.

\begin{notation}\label{nota:reduction}
	Let $X$ be a $\tau^{-1}_\M$-rigid module. We denote by $t_X$ the torsion functor associated with the torsion pair $({}^\perp X, \Cogen X)$ from Proposition~\ref{prop:ff_torsion}.
\end{notation}

We recall the following relationship between $\tau^{-1}_\M$-rigid modules and $\tau^{-1}_{\Jinv_\M(X)}$-rigid modules. For the convenience of the reader, we outline a proof explaining how to deduce these statements from those in \cite{BM_exceptional}.

\begin{proposition}\cite[Props.~4.5 and~5.8]{BM_exceptional}\textnormal{(see also \cite[Thm.~3.14]{jasso})}\label{prop:induction_TF}
	Let $X$ be a $\tau^{-1}_\M$-rigid module. Then the association $(Y,X) \mapsto t_X Y$ is a bijection from the set of weakly cogen-preordered $\tau^{-1}_\M$-rigid modules of the form $(Y,X)$ to the set of $\tau^{-1}_{\Jinv_\M(X)}$-rigid modules.
\end{proposition}

\begin{proof}
	The dual of \cite[Prop.~4.5]{BM_exceptional} says that $t_X$ induces a bijection from the set of indecomposable modules $Y$ for which $(Y,X)$ is a weakly cogen-preordered $\tau^{-1}_\M$-rigid module to the set of indecomposable $\tau^{-1}_{\Jinv_\M(X)}$-rigid modules. In the dual of \cite[Prop.~5.8]{BM_exceptional}, this is extended into a direct-sum-preserving bijection $\mathcal{E}^{\mathsf d}_X$ between the set of objects $Y$ in the bounded derived category $\mathcal{D}^b(\M)$ for which $X \oplus Y$ is a $\tau^{-1}_\M$-rigid \emph{pair} and the set of $\tau^{-1}_{\Jinv_\M(X)}$-rigid \emph{pairs}. Now given an object $Y$ in the domain of $\mathcal{E}^{\mathsf d}_X$, one has that $\mathcal{E}^{\mathsf d}_X(Y)$ is a module if and only if $(Y,X)$ is a weakly cogen-preordered $\tau^{-1}_\M$-rigid module, and that $\mathcal{E}^{\mathsf d}_X(Y) = t_X Y$ in this case. (The last equality uses the additivity of $t_X$.)
\end{proof}

Propositions~\ref{prop:AP} and~\ref{prop:induction_TF} are related by the following. See the proof of Proposition~\ref{prop:induction_TF}
when comparing the notation of Proposition~\ref{prop:tors_tilt_compat}(2) with that of \cite[Lem.~6.6]{BuH}.

\begin{proposition}\label{prop:tors_tilt_compat}
	Let $(Y,X)$ be a weakly cogen-preordered $\tau^{-1}_\M$-rigid module. Then
	\begin{enumerate}
		\item \cite[Lem.~5.2]{MT} $\Cogen(X \oplus Y) \cap \Jinv_\M(X) = \Cogen(t_X Y) \cap \Jinv_\M(X)$.
		\item \cite[Lem.~6.6]{BuH} $(\tau^{-1}_\M (X\oplus Y))^\perp \cap \Jinv_\M(X) = (\tau^{-1}_{\Jinv_\M(X)} t_X Y)^\perp \cap \Jinv_\M(X)$.
	\end{enumerate}
\end{proposition}

The following essentially follows from \cite[Cor.~5.3]{MT}, but we outline a proof for the convenience of the reader.

\begin{corollary}\label{cor:induction_TF}
    Let $\Delta = (X_1,\ldots,X_m)$ be a cogen-preordered $\tau^{-1}_\M$-rigid module. Then $t_{\Delta_{>1}}X_1$ is a cogen-minimal $\tau^{-1}_{\Jinv_\M(\Delta_{>1})}$-rigid module.
\end{corollary}

\begin{proof}
    We prove the contrapositive. Denote $X = t_{\Delta_{>1}}X_1$. We note  that $X$ is a $\tau^{-1}_{\Jinv_\M(\Delta_{>1})}$-rigid module by Proposition~\ref{prop:induction_TF}.  Suppose that $X$ is not cogen-minimal. Then there exists an indecomposable direct summand $Z$ of $X$ such that $Z \in \Cogen(X/\add(Z))\cap \Jinv_\M(\Delta_{>1})$. Equivalently (see Remark~\ref{rem:cogen}), $\Cogen X\cap \Jinv_\M(\Delta_{>1}) = \Cogen(X/\add(Z)) \cap \Jinv_\M(\Delta_{>1})$. Now by Proposition~\ref{prop:induction_TF}, there exists an indecomposable direct summand $Z'$ of $X_1$ such that $t_{\Delta_{> 1}}Z' = Z$ and $t_{\Delta_{> 1}} (X_1/\add Z') = X/\add Z$. Proposition~\ref{prop:tors_tilt_compat}(1) then says that $\Cogen(\Delta_{\geq 1}) = \Cogen(\Delta_{\geq 1}/\add Z')$. Equivalently (see Remark~\ref{rem:cogen}), $Z' \in \Cogen(\Delta_{\geq 1}/\add Z')$. We conclude that $\Delta$ is not cogen-preordered.
\end{proof}


\section{Main results}\label{sec:intervals}

We now prove the main results of this paper.  We start with the following, which in particular proves the $m = 1$ case of Theorem~\ref{thm:mainA}.

\begin{proposition}\label{prop:meet}
	The association $X \mapsto [\Cogen X, (\tau^{-1}_\M X)^\perp]$ is a bijection from the set of cogen-minimal $\tau^{-1}_\M$-rigid modules to the set of functorially finite maximal join intervals in $\torf\M$. The inverse is given by $\sI \mapsto \Imods(\sI^-)$.
\end{proposition}

\begin{proof}
	\cite[Prop.~8.4]{BaH} says that $[\Cogen X, (\tau^{-1}_\M X)^\perp]$ is a functorially finite maximal join interval for any cogen-minimal $\tau^{-1}_\M$-rigid module $X$. The fact that this association is a bijection with the indicated inverse then follows from Proposition~\ref{prop:cogen_min_ff}.
\end{proof}

We also need the following refinement of \cite[Prop.~5.11]{sakai}.

\begin{lemma}\label{lem:meet_int_restrict}
	Let $X$ be a $\tau^{-1}_\M$-rigid module and let $\sI \subseteq [\Cogen X, (\tau^{-1}_\M X)^\perp]$ be an interval. Then the following are equivalent.
    \begin{enumerate}
        \item $\sI$ is a functorially finite maximal join interval in $[\Cogen X, (\tau^{-1}_\M X)^\perp]$.
        \item $[\sI^- \cap \Jinv_\M(X), \sI^+ \cap \Jinv_\M(X)]$ is a functorially finite maximal join interval in $\torf \Jinv_\M(X)$.
    \end{enumerate}
\end{lemma}

\begin{proof}
	We have that $[\sI^- \cap \J_\M(X), \sI^+ \cap \J_\M(X)]$ is functorially finite if and only if $\sI$ is functorially finite by \cite[Thm.~3.13]{jasso} (see also Lemma~\ref{lem:ff_ints}). We have that $[\sI^- \cap \J_\M(X), \sI^+ \cap \J_\M(X)]$ is a maximal join interval in $\torf \J_\M(X)$ if and only if $\sI$ is a maximal join interval in $[\Cogen X, (\tau^{-1}_\M X)^\perp]$ by \cite[Prop.~5.11]{sakai}.
\end{proof}

We now construct the bijections comprising Theorem~\ref{thm:mainA} (restated as Theorem~\ref{thm:meet_and_TF} below).

\begin{proposition}\label{prop:TF_to_meet}
    Let $\Delta = (X_1,\ldots,X_m)$ be a cogen-preordered $\tau^{-1}_\M$-rigid module. For $1 \leq j \leq m$, denote $\F_j = \Cogen \Delta_{\geq j}$ and $\G_j = (\tau^{-1}_\M \Delta_{\geq j})^\perp$. Then $\phi(\Delta) := ([\F_1,\G_1],\ldots,[\F_m,\G_m])$ is a decreasing sequence of functorially finite maximal join intervals in $\torf\M$. 
\end{proposition}

\begin{proof}
    We prove the result by induction on $m$. For $m = 0$, we have that $\phi$ sends the empty sequence to the empty sequence, so there is nothing to show. Thus suppose $m > 0$ and that the result holds for $m - 1$. Let $\Delta = (X_1,\ldots,X_m)$ be a cogen-preordered $\tau^{-1}_\M$-rigid module of length $m$ and denote $X := \Delta_{> 1}$. By the induction hypothesis, we need only show that $[\F_1,\G_1]$ is a functorially finite maximal join interval in $[\F_2,\G_2]$.

    By Corollary~\ref{cor:induction_TF}, we have that $t_X Y$ is a cogen-minimal $\tau^{-1}_{\Jinv_\M(X)}$-rigid module. Thus, by Proposition~\ref{prop:meet}, $[\Cogen t_X Y \cap \Jinv_\M(X), (\tau_\M^{-1} t_X Y)^\perp \cap \Jinv_\M(X)]$ is a functorially finite maximal join interval in $\torf \Jinv_\M(X)$. Moreover, we have that $\Cogen t_X Y \cap \Jinv_\M(X) = \F_1 \cap \Jinv_\M(X)$ and that $(\tau_\M^{-1} t_X Y)^\perp \cap \Jinv_\M(X) = \G_1 \cap \Jinv_\M(X)$ by Proposition~\ref{prop:tors_tilt_compat}. Since $[\F_1,\G_1] \subseteq [\F_2,\G_2]$ by construction, it follows from Lemma~\ref{lem:meet_int_restrict} that $[\F_1,\G_1]$ is a functorially finite maximal join interval in $[\F_2,\G_2]$, as desired.
\end{proof}

\begin{proposition}\label{prop:meet_to_TF}
    Let $\mathfrak{J} = (\sI_1,\ldots,\sI_m)$ be a decreasing sequence of functorially finite maximal join intervals in $\torf \M$, and denote $\sI_{m+1} = [0,\M]$. For $1 \leq k \leq m$, denote
    $$X_k = (\Imods(\sI_k^-))/\sI_{k+1}^-.$$
    Then $\psi(\mathfrak{J}) := (X_1,\ldots,X_m)$ is a cogen-preordered $\tau^{-1}_\M$-rigid module. Moreover, $$\sI_k = \left[\Cogen (\psi(\mathfrak{J})_{\geq k}), \left(\tau^{-1}_\M (\psi(\mathfrak{J})_{\geq k})\right)^\perp\right]$$ for all $1 \leq k \leq m$.
\end{proposition}

\begin{proof}
    We prove the result by reverse induction on $m$. For $m = 0$, we have that $\mathfrak{J}$ and $\psi(\mathfrak{J})$ are both the empty sequence, so there is nothing to show. Thus suppose that $m > 0$ and that the result holds for $m-1$.

    Let $\mathfrak{J} = (\sI_1,\ldots,\sI_m)$ be a decreasing sequence of functorially finite maximal join intervals in $\torf \M$ and denote $X:= \psi(\mathfrak{J})_{> 1}$. By the induction hypothesis, we have that $X$ is $\tau^{-1}_\M$-rigid, that $(X_2,\ldots,X_m)$ is cogen-preordered, and that $\sI_2 = [\Cogen X, (\tau^{-1}_\M X)^\perp]$. Now by assumption, we have $\sI_2^- \subseteq \sI_1^- \subseteq \sI_2^+$. Lemma~\ref{lem:sum_tau_rigid} thus says that $X \oplus \Imods(\sI_1^-)$ is $\tau_\M^{-1}$-rigid and that $\sI_1^- = \Cogen(X \oplus \Imods(\sI_1^-))$. Moreover, we have that $X_1$ is cogen-minimal since it is a direct summand of $\Imods(\sI_1^-)$. Now, by construction, any indecomposable direct summand of $\Imods(\sI_1^-)$ is either a direct summand of $X_1$ or lies in $\Cogen X$. Thus we have (i) that $\Cogen(\bigoplus \psi(\mathfrak{J})) = \Cogen(X_1 \oplus X) = \sI_1^-$, and (ii) that $(X_1,X)$ is cogen-preordered. The induction hypothesis then implies that $\psi(\mathfrak{J}) = (X_m,\ldots,X_1)$ is cogen-preordered.

    It remains only to show that $\sI_1^+ = (\tau_{\M}^{-1} (X\oplus X_1))^\perp$. To see this, note, by Lemma~\ref{lem:meet_int_restrict} (and the induction hypothesis), that $[\sI_1^- \cap \Jinv_\M(X), \sI_1^+ \cap \Jinv_\M(X)]$ is a functorially finite maximal join interval in $\torf \Jinv_\M(X)$. Now $t_X X_1$ is a cogen-minimal $\tau^{-1}_{\Jinv_\M(X)}$-rigid module by Corollary~\ref{cor:induction_TF}, and $\sI_1^- \cap \Jinv_\M(X) = \Cogen t_X X_1 \cap \Jinv_\M(X)$ by Proposition~\ref{prop:tors_tilt_compat}(1). Thus, by Proposition~\ref{prop:meet}, we have $\sI_1^+ \cap \Jinv_\M(X) = (\tau_{\Jinv_\M(X)}^{-1} t_X X_1)^\perp \cap \Jinv_\M(X)$. Noting that $\Cogen X = \sI_2^- \subseteq \sI_1^+ \subseteq \sI_2^+ = (\tau^{-1}_\M X)^\perp$, it then follows from Propositions~\ref{prop:tors_tilt_compat}(2) and Proposition~\ref{prop:AP} that $\sI_1^+ = (\tau_\M(X\oplus X_1))^\perp$.
\end{proof}

We are now prepared to prove our main theorem.

\begin{theorem}[Theorem~\ref{thm:mainA}]\label{thm:meet_and_TF}
    Let $m$ be a nonnegative integer. Then the maps $\phi$ and $\psi$ from Propositions~\ref{prop:TF_to_meet} and~\ref{prop:meet_to_TF} are inverse bijections between
    \begin{itemize}
        \item[(1)] The set of isomorphism classes of cogen-preordered basic $\tau_\M^{-1}$-rigid modules of length $m$, and
        \item[(2)] The set of decreasing sequences of maximal join intervals of length $m$ in $\torf(\mods\Lambda)$.
    \end{itemize}
\end{theorem}

\begin{proof}
    The fact that the maps are well-defined is contained Propositions~\ref{prop:TF_to_meet} and~\ref{prop:meet_to_TF}. The fact that $\phi \circ \psi$ is the identity also follows immediately from these propositions. Thus let $\Delta = (X_1,\ldots,X_m)$ be a cogen-preordered $\tau^{-1}_\M$-rigid module and write $\psi \circ \phi(\Delta) = (Y_1,\ldots,Y_m)$. Choose an index $1 \leq k \leq m$. Then $Y_k = (\Imods(\Cogen \Delta_{\geq k}))/\Cogen \Delta_{> k}$ by construction and $X_k = (\Imods(\Cogen\Delta_{\geq k}))/\Cogen \Delta_{> k}$ by Lemma~\ref{lem:cogen_preorder}. We conclude that $X_k = Y_k$.
\end{proof}


 The next two results use Sakai's correspondence between 
 ICE-sequences and decreasing sequences of maximal join intervals (see Section~\ref{subsec:ICE_ints}) to recast Theorem~\ref{thm:meet_and_TF} as a classification of contravariantly finite ICE-sequences via cogen-preordered $\tau^{-1}_\M$-rigid modules.

\begin{corollary}[Corollary~\ref{cor:mainB}, part 1]\label{cor:TF_to_ICE}
    Let $m$ be a nonnegative integer. Then the map $\nu \circ \phi$ (see Propositions~\ref{prop:meet_to_ICE} and~\ref{prop:TF_to_meet} for the definitions) is a bijection from the set of isomorphism classes of cogen-preordered basic $\tau_\M^{-1}$-rigid modules of length $m$ to the set of contravariantly finite ICE-sequences of length $m+1$. Moreover, let $\Delta = (X_1,\ldots,X_m)$ be a cogen-preordered $\tau^{-1}_\M$-rigid module. For $k \in \mathbb{Z}$, denote
    $$\C(k) = \begin{cases} 0 & \text{if $k > 0$}\\
        (\tau^{-1}_\M \Delta_{> 1-k})^\perp \cap {}^\perp \Delta_{\geq 1-k} & \text{if $-m < k \leq 0$}\\
        \M & \text{if $k \leq -m$}.
        \end{cases}$$
        Then $\nu\circ \phi(\Delta) = (\C(k))_{k \in \mathbb{Z}}$.
\end{corollary}

\begin{proof}
    The explicit formula and the fact $\nu\circ\phi$ is a bijection follow immediately from Proposition~\ref{prop:meet_to_ICE}, Lemma~\ref{lem:meet_to_ICE_ff}, Proposition~\ref{prop:TF_to_meet}, and Theorem~\ref{thm:meet_and_TF}.
\end{proof}

The next result gives the inverse of $\nu \circ \phi$. Recall the definition of the map $\psi$ (which is the inverse of $\phi$) from Proposition~\ref{prop:meet_to_TF}.

\begin{corollary}[Corollary~\ref{cor:mainB}, part 2]\label{cor:ICE_to_TF}
    Let $m$ be a nonnegative integer and let $(\C(k))_{k \in \mathbb{Z}}$ be a contravariantly finite ICE-sequence of length $m+1$. For $1 \leq k \leq m+1$, let $\F_k$ be as in Proposition~\ref{prop:ICE_to_meet}. Then $\psi\circ\nu^{-1}((\C(k))_{k \in \mathbb{Z}}) = (X_1,\ldots,X_m)$ where $X_k = (\Imods(\F_{k}))/\F_{k-1}$ for $1 \leq k \leq m$.
\end{corollary}

\begin{proof}
    Write $\nu^{-1}((\C(k))_{k \in \mathbb{Z}} = (\sI_1,\ldots,\sI_m)$. Then, for $1 \leq k \leq m$, we have $\sI_m^- = \F_k$ by Proposition~\ref{prop:ICE_to_meet}. The result then follows from the definition of $\psi$ (see Proposition~\ref{prop:meet_to_TF}).
\end{proof}

We conclude by using \cite[Cor.~5.6]{sakai} to recast Corollary~\ref{cor:TF_to_ICE} as a classification of $(m+1)$-intermediate $t$-structures whose aisles are homology determined. We refer to \cite{sakai} for the relevant definitions.

\begin{corollary}[Corollary~\ref{cor:mainC}]\label{cor:TF_to_t_str}
    Let $m$ be a nonnegative integer. Then there is a bijection $\rho$ from the set of isomorphism classes of cogen-preordered basic $\tau_\M^{-1}$-rigid modules of length $m$ to the set of $(m+1)$-intermediate $t$-structures in $\derD(\M)$ whose aisles are homology-determined given as follows. Let $\Delta = (X_1,\ldots,X_m)$ be a cogen-preordered $\tau^{-1}_\M$-rigid module and let $(\C(k))_{k \in \mathbb{Z}}$ be as in Corollary~\ref{cor:TF_to_ICE}. Then $\rho(\Delta) = (\B^{\leq 0},\B^{\geq 0})$ where
    \begin{align*}
        \B^{\leq 0} &= \{U \in \derD(\M) \mid \forall k \in \mathbb{Z}: H^k(U) \in \C(k)\},\\
        \B^{\geq 0} &= \{V \in \derD(\M) \mid \forall U \in \B^{\leq 0}: \Hom(U[1],V) = 0\}.
    \end{align*}
\end{corollary}

\begin{proof}
    The formula for $\B^{\leq 0}$ follows from combining the formula for $\nu\circ\phi$ (Corollary~\ref{cor:TF_to_ICE}) with the formula for the bijection in \cite[Cor.~5.6]{sakai} (given explicitly in \cite[Prop.~3.7]{sakai} and \cite[Prop.~4.10]{SvR}). The fact that $\rho(\Delta)$ is an $(m+1)$-intermediate $t$-structure whose aisle is homology determined, and the fact that $\rho$ is a bijection, then follow from \cite[Cor.~5.6]{sakai} and Corollary~\ref{cor:ICE_to_TF} (see also \cite[Sec.~1, Prop.]{KV}).
\end{proof}


\bibliographystyle{amsalpha}
\bibliography{biblio.bib}

\newcommand{\etalchar}[1]{$^{#1}$}
\providecommand{\bysame}{\leavevmode\hbox to3em{\hrulefill}\thinspace}
\providecommand{\MR}{\relax\ifhmode\unskip\space\fi MR }
\providecommand{\MRhref}[2]{%
  \href{http://www.ams.org/mathscinet-getitem?mr=#1}{#2}
}
\providecommand{\href}[2]{#2}
\begin{thebibliography}{DIR{\etalchar{+}}23}

\bibitem[AIR14]{AIR}
T.~Adachi, O.~Iyama, and I.~Reiten, \emph{{$\tau$}-tilting theory}, Compos.
  Math. \textbf{150} (2014), no.~3, 415--452.

\bibitem[AP22]{AP}
S.~Asai and C.~Pfeifer, \emph{Wide subcategories and lattices of torsion
  classes}, Algebr. Represent. Theory \textbf{25} (2022), 1611--1629.

\bibitem[AS81]{AS}
M.~Auslander and S.~O. Smal\o, \emph{Almost split sequences in subcategories},
  J. Algebra \textbf{69} (1981), no.~2, 426--454.

\bibitem[ASS06]{ASS}
I.~Assem, D.~Simson, and A.~Skowroński, \emph{Elements of the representation
  theory of associative algebras, volume 1: techniques of representation
  theory}, London Math. Soc. Stud. Texts, vol.~65, Cambridge University Press,
  Cambridge, 2006.

\bibitem[BDH]{BDH}
E.~Barnard, C.~Defant, and E.~J. Hanson, \emph{Pop-stack operators for torsion
  classes and {C}ambrian lattices}, arXiv:2312.03959 [math.CO].

\bibitem[BH23]{BuH}
A.~B. Buan and E.~J. Hanson, \emph{$\tau$-perpendicular wide subcategories},
  Nagoya Math. J. \textbf{252} (2023), 959--984.

\bibitem[BH24]{BaH}
E.~Barnard and E.~J. Hanson, \emph{Exceptional sequences in semidistributive
  lattices and the poset topology of wide subcategories}, J. Algebra Appl.
  (2024).

\bibitem[BHM]{BMH}
A.~B. Buan, E.~J. Hanson, and B.~R. Marsh, \emph{Mutation of $\tau$-exceptional
  pairs and sequences}, arXiv:2402.10301 [math.RT].

\bibitem[BM21]{BM_exceptional}
A.~B. Buan and B.~R. Marsh, \emph{$\tau$-exceptional sequences}, J. Algebra
  \textbf{585} (2021), 36--68.

\bibitem[BR07]{BR}
A.~Beligiannis and I.~Reiten, \emph{Homological and homotopical aspects of
  torsion theories}, Mem. Amer. Math. Soc. \textbf{188} (2007), no.~883.

\bibitem[DF15]{DF}
H.~Derksen and J.~Fei, \emph{General presentations of algebras}, Adv. Math,
  \textbf{278} (2015), 210--237.

\bibitem[DIR{\etalchar{+}}23]{DIRRT}
L.~Demonet, O.~Iyama, N.~Reading, I.~Reiten, and H.~Thomas, \emph{Lattice
  theory of torsion classes: beyond $\tau$-tilting theory}, Trans. Amer. Math.
  Soc. Ser. B \textbf{10} (2023), 542--612.

\bibitem[Eno22]{enomoto_rigid}
H.~Enomoto, \emph{Rigid modules and {ICE}-closed subcategories in quiver
  representations}, J. Algebra \textbf{594} (2022), 364--388.

\bibitem[ES22]{ES}
H.~Enomoto and A.~Sakai, \emph{{ICE}-closed subcategories and wide
  $\tau$-tilting modules}, Math. Z. \textbf{300} (2022), 541--577.

\bibitem[Han24]{facial}
E.~J. Hanson, \emph{A facial order for torsion classes}, Int. Math. Res. Not.
  IMRN \textbf{2024} (2024), no.~12, 9849–9874.

\bibitem[HRS96]{HRS}
D.~Happel, I.~Reiten, and S.O. Smal{\o}, \emph{Tilting in abelian categories
  and quasitilted algebras}, Mem. Amer. Math. Soc. \textbf{120} (1996),
  no.~575.

\bibitem[IRTT15]{IRTT}
O.~Iyama, I.~Reiten, H.~Thomas, and G.~Todorov, \emph{Lattice structure of
  torsion classes for path algebras}, B. Lond. Math. Soc. \textbf{47} (2015),
  no.~4, 639--650.

\bibitem[IT09]{IT}
C.~Ingalls and H.~Thomas, \emph{Noncrossing partitions and representations of
  quivers}, Compos. Math. \textbf{145} (2009), no.~6, 1533--1562.

\bibitem[Jas14]{jasso}
G.~Jasso, \emph{Reduction of {$\tau$}-tilting modules and torsion pairs}, Int.
  Math. Res. Not. IMRN \textbf{2015} (2014), no.~16, 7190--7237.

\bibitem[KV88]{KV}
B.~Keller and D.~Vossieck, \emph{Aisles in derived categories}, Bull. Soc.
  Math. Belg. \textbf{40} (1988), 239--253.

\bibitem[M{\v S}17]{MS}
F.~Marks and J.~{\v S}{\v t}ov\'i{\v c}ek, \emph{Torsion classes, wide
  subcategories, and localisations}, Bull. Lond. Math. Soc. \textbf{49} (2017),
  no.~3.

\bibitem[MT20]{MT}
H.~O. Mendoza and H.~Treffinger, \emph{Stratifying systems through
  $\tau$-tilting theory}, Doc. Math. \textbf{25} (2020), 701--720.

\bibitem[Sak]{sakai}
A.~Sakai, \emph{Classifying $t$-structures via {ICE}-closed subcategories and a
  lattice of torsion classes}, arXiv:2307.11347 [math.RT].

\bibitem[Sma84]{smalo}
S.~O. Smal\o, \emph{Torsion theories and tilting modules}, Bull. London Math.
  Soc. \textbf{16} (1984), no.~5, 518--522.

\bibitem[SvR19]{SvR}
D.~Stanley and A.-C. van Roosmalen, \emph{$t$-structures on hereditary
  categories}, Math. Z. \textbf{293} (2019), no.~1-2, 731--766.

\bibitem[Tat21]{tattar}
A.~Tattar, \emph{Torsion pairs and quasi-abelian categories}, Algebr.
  Represent. Theory \textbf{24} (2021), 1557--1581.

\bibitem[Tho21]{thomas_intro}
H.~Thomas, \emph{An introduction to the lattice of torsion classes}, Bull.
  Iranian Math. Soc. \textbf{47} (2021), 35--55.

\bibitem[Woo10]{woolf}
J.~Woolf, \emph{Stability conditions, torsion theories, and tilting}, J. Lond.
  Math. Soc. (2) \textbf{82} (2010), no.~3, 663--682.

\end{thebibliography}

\end{document}